\documentclass{article}%
\usepackage{amssymb}
\usepackage{amsmath}
\usepackage{amsfonts}
\usepackage{graphicx}%
\providecommand{\U}[1]{\protect \rule{.1in}{.1in}}
\newtheorem{theorem}{Theorem}

\newtheorem{corollary}[theorem]{Corollary}

\newtheorem{proposition}[theorem]{Proposition}

\newenvironment{proof}[1][Proof]{\noindent \textbf{#1.} }{\  \rule{0.5em}{0.5em}}
\begin{document}

\title{Characterizations of special skew ruled surfaces \\by the normal
curvature of some \\distinguished families of curves}
\author{Stylianos Stamatakis\strut \medskip\\ \emph{Aristotle University of Thessaloniki}\\ \emph{Department of Mathematics}\\ \emph{GR-54124 Thessaloniki, Greece}\\ \emph{e-mail: stamata@math.auth.gr}}
\date{}
\maketitle

\begin{abstract}
\noindent We consider skew ruled surfaces in the three-dimensional Euclidean
space and some geometrically distinguished families of curves on them whose
normal curvature has a concrete form. The aim of this paper is to find and
classify all ruled surfaces with the mentioned property\smallskip
.\vspace{0.1cm}\newline MSC 2000: 53A25, 53A05\newline Keywords: Ruled
surfaces, normal curvature, Edlinger-surfaces, right helicoid

\end{abstract}

\section{Introduction}

Geometrically distinguished families of curves on a skew ruled surface in the
Euclidean space $%
\mathbb{R}
^{3}$ have been studied by a number of authors and in many points of view. A
range of results appears when one requires that the curves of the considered
family possess an additional property. The present paper contributes to this
field of themes. We consider special families of curves on a skew ruled
surface and suppose that the normal curvature along these curves has a
concrete form. Our aim is to find the type of all ruled surfaces with the
mentioned property and to classify them. The results are assembled in the
table at the end of the paper.\smallskip

To set the stage for this work the classical notation of ruled surface theory
is briefly presented; for this purpose \cite{Hoschek} is used as a general
reference. In the Euclidean space $%
\mathbb{R}
^{3}$ let $\Phi$ be a regular ruled surface without torsal rulings determined
on $G:=I\times%
\mathbb{R}
$ ($I\subset%
\mathbb{R}
$ open interval) and of the class $C^{3}$. $\Phi$ can be expressed in terms of
the striction line $\boldsymbol{s}=\boldsymbol{s}(u)$ and the unit vector
field $\boldsymbol{e}(u)$ pointing along the rulings as%
\begin{equation}
\boldsymbol{x}(u,v)=\boldsymbol{s}(u)+v\, \boldsymbol{e}(u),\;u\in I,\;v\in%
\mathbb{R}
.\label{1}%
\end{equation}
Moreover we can choose the parameter $u$ to be the arclength along the
spherical curve $\boldsymbol{e}(u)$. Putting $f\,%
\acute{}%
\,(u)=\frac{df}{du}$ for a differentiable function $f(u)$ we have
\begin{equation}
\langle \boldsymbol{s}\,%
\acute{}%
\,(u),\, \boldsymbol{e}\,%
\acute{}%
\,(u)\rangle=0,\quad \left \vert \boldsymbol{e}\,%
\acute{}%
\,(u)\right \vert =1\quad \forall \;u\in I,\label{2}%
\end{equation}
where $\langle$ , $\rangle$ denotes the standard inner product in $%
\mathbb{R}
^{3}$. The \emph{conical curvature }$k(u)$, the \emph{parameter of
distribution }$\delta(u)$ and the \emph{striction }$\sigma(u)$ of the surface
$\Phi$ are given by
\[
k(u)=(\boldsymbol{e}(u),\boldsymbol{e}\,%
\acute{}%
\,(u),\boldsymbol{e}\,%
\acute{}%
\, \,%
\acute{}%
\,(u)),\  \delta(u)=(\boldsymbol{e}(u),\boldsymbol{e}\,%
\acute{}%
\,(u),\boldsymbol{s}\,%
\acute{}%
\,(u)),\  \sigma(u):=\sphericalangle(\boldsymbol{e}(u),\boldsymbol{s}\,%
\acute{}%
\,(u)),
\]
where
\[
-\frac{\pi}{2}<\sigma \leq \frac{\pi}{2}\quad \text{and\quad}\operatorname*{sign}%
\sigma=\operatorname*{sign}\delta.
\]
The functions $k(u)$, $\delta(u)$ and $\sigma(u)$ consist a complete system of
invariants of the surface $\Phi$ (\cite{Hoschek}; p.19).\medskip \newline The
components $g_{ij}$ and $h_{ij}$ of the first and the second fundamental
tensors in (local) coordinates $u^{1}:=u$, $u^{2}:=v$ are the following%
\begin{equation}
\left \{
\begin{array}
[c]{c}%
{\normalsize (g}_{ij}{\normalsize )=}\left(
\begin{array}
[c]{cc}%
v^{2}+\delta^{2}\left(  \lambda^{2}+1\right)   & \delta \, \lambda \\
\delta \, \lambda & 1
\end{array}
\right)  \\
{\small (h}_{ij}{\small )=}\frac{1}{w}\left(
\begin{array}
[c]{cc}%
-\left[  k\,v^{2}{\small +}\delta \,%
\acute{}%
v{\small +}\delta^{2}\left(  k-\lambda \right)  \right]   & \delta \\
\delta & 0
\end{array}
\right)
\end{array}
\right.  ,\label{3}%
\end{equation}
where $w:=\sqrt{v^{2}+\delta^{2}}$ and $\lambda:=\cot \sigma$. The Gaussian
curvature $K$ and the mean curvature $H$ of $\Phi$ are given respectively by%
\begin{equation}
K=\frac{-\delta^{2}}{w^{4}},\quad H=-\frac{k\,v^{2}+\delta \,%
\acute{}%
\,v+\delta^{2}\left(  k+\lambda \right)  }{2\,w^{3}}.\label{4}%
\end{equation}

Skew ruled surfaces $\Phi$, whose osculating quadrics are rotational
hyperboloids, are called \emph{Edlinger-surfaces} \cite{Edlinger},
\cite{Hoschek}. Necessary and sufficient conditions for a ruled surface $\Phi$
to be an Edlinger-surface are the following (\cite{Brauner}; p.103):%
\[
\delta \,%
\acute{}%
\,=k\, \lambda+1=0.
\]
This is a ruled surface of constant parameter of distribution whose striction
line $\boldsymbol{s}(u)$ is a line of curvature. The curves of \emph{constant
striction distance}, i.e. the curves $v=constant$, are in this case lines of
curvature of $\Phi$. The other family of the lines of curvature is determined
by%
\[
\lbrack k^{2}\,v^{2}+\delta^{2}(k^{2}+1)]\,du-\delta \,k\,dv=0.
\]
It is easily verified that the corresponding normal curvatures of the lines of
curvature (principal curvatures) are the following%
\begin{equation}
k_{1}=-k(u)\,w^{-1},\quad k_{2}=\frac{\delta^{2}(u)}{k(u)}w^{-3}. \label{5}%
\end{equation}

In the rest of this paper \emph{only skew (non-developable) ruled surfaces of
the space}\textit{ }$%
\mathbb{R}
^{3}$ are considered \emph{with the parametrization} (\ref{1})
\emph{satisfying the conditions} (\ref{2})\textit{.}

\section{The case of the principal curvatures}

Starting point of this paragraph are the relations (\ref{5}). Firstly the
problem \emph{of finding all ruled surfaces whose a principal curvature has
the following form} is considered%
\begin{equation}
k_{i}=f(u)\,w^{n},\;n\in%
\mathbb{Z}
,\;f(u)\in C^{0}(I),\;i=1\; \text{or}\;2. \label{6}%
\end{equation}
It is obvious that $f(u)\neq0\;$for all $u\in I$ since $\Phi$ is
non-developable.\smallskip

\noindent Using (\ref{3}) the normal curvature in direction $du:dv$ is found
to be%
\begin{equation}
k_{N}=\frac{1}{w}\cdot \frac{-\left[  kv^{2}+\delta \,%
\acute{}%
\,v+\delta^{2}\left(  k-\lambda \right)  \right]  \,du^{2}+2\, \delta
\,du\,dv}{\left[  v^{2}+\delta^{2}\left(  \lambda^{2}+1\right)  \right]
\,du^{2}+2\, \delta \, \lambda \,du\,dv+dv^{2}}\text{.} \label{7}%
\end{equation}
Taking into account (\ref{6}) it follows that%
\begin{align*}
&  [f\,w^{n+1\,}[v^{2}+\delta^{2}(\lambda^{2}+1)]+k\,v^{2}+\delta \,%
\acute{}%
\,v+\delta^{2}(k-\lambda)]\,du^{2}\\
&  +2\, \delta \,(f\, \lambda \,w^{n+1}-1)\,du\,dv+f\,w^{n+1\,}dv^{2}=0.
\end{align*}
This second order in $du:dv$ equation has \emph{exactly one solution }if and
only if its discriminant vanishes:%
\begin{equation}
f^{2}\,w^{2n+4}+f\,[k\,v^{2}+\delta \,%
\acute{}%
\,v+\delta^{2}(k+\lambda)]\,w^{n+1}-\delta^{2}=0\quad \forall \;u\in I,\;v\in%
\mathbb{R}
. \label{8}%
\end{equation}
Let first be $n=0$. Then (\ref{8}) changes into%
\[
f^{4}(v^{2}+\delta^{2})^{4}-2\,f^{2}\delta^{2}(v^{2}+\delta^{2})^{2}%
+\delta^{4}-f^{2\,}[k\,v^{2}+\delta \,%
\acute{}%
\,v+\delta^{2}(k+\lambda)]^{2}(v^{2}+\delta^{2})=0.
\]
In the left hand stays a polynomial of degree eight in $v$ which vanishes for
all $u\in I$ and infinite $v\in%
\mathbb{R}
$. Comparing its coefficients with those of the zero polynomial it becomes
obvious that $f$ vanishes, which it was previously excluded.\smallskip \newline
We distinguish now the following cases:\smallskip \newline \emph{Case I}: Let
$n\in%
\mathbb{Z}
$ in (\ref{8}) be odd. Then we have%
\begin{align*}
Q(v):  &  =f\,[k\,v^{2}+\delta \,%
\acute{}%
\,v+\delta^{2}(k+\lambda)](v^{2}+\delta^{2})^{\frac{n+1}{2}}\\
&  +f^{2}(v^{2}+\delta^{2})^{n+2}-\delta^{2}=0\quad \forall \;u\in I,\;v\in%
\mathbb{R}
.
\end{align*}
For $n\geq1$ the vanishing of the coefficient of $v^{2\left(  n+2\right)  }$,
which is the greatest power of $v$ of the polynomial $Q(v),$ implies $f=0$,
which is impossible.\smallskip \newline Let $n=-1$. Then
\[
Q(v)=f^{2}(v^{2}+\delta^{2})+f\,[k\,v^{2}+\delta \,%
\acute{}%
\,v+\delta^{2}(k+\lambda)]-\delta^{2}=0.
\]
The vanishing of the coefficients of the polynomial $Q(v)$ gives%
\[
f=-k,\quad \delta \,%
\acute{}%
\,=k\, \lambda+1=0,
\]
therefore $\Phi$ is an Edlinger-surface.\smallskip \newline

\noindent Let $n=-3$ and $k_{1}$ the principal curvature having the form
(\ref{6}), i.e. $k_{1}=f(u)\,w^{-3}$. Then from (\ref{4}) for the other
principal curvature the following expression is obtained%

\[
k_{2}=f^{\ast}(u)\,w^{-1}\text{\quad with\quad}f^{\ast}(u):=\frac{-\delta
^{2}(u)}{f(u)},
\]
so that a principal curvature of $\Phi$ has the form (\ref{6}), where $n=-1$.
As it was previously established $\Phi$ is an Edlinger-surface.\smallskip
\newline The case $n\leq-5$ leads to a contradiction as one can easily
confirm.\medskip \newline \emph{Case II}: Let $n\in%
\mathbb{Z}
$ in (\ref{8}) be even. For $n=-2$ it follows from (\ref{8})%
\[
Q(v):=f^{2}[k\,v^{2}+\delta \,%
\acute{}%
\,v+\delta^{2}(k+\lambda)]^{2}-(f^{2}-\delta^{2})^{2}(v^{2}+\delta^{2}%
)=0\quad \forall \;u\in I,\;v\in%
\mathbb{R}
.
\]
The vanishing of the coefficient $f^{2}k^{2}$ of $v^{4}$ implies $k=0$. From
the vanishing of the remaining coefficients of the polynomial $Q(v)$ the
outcome is%
\[
f^{2}\, \delta \,%
\acute{}%
\,^{2}-(f^{2}-\delta^{2})^{2}=2f^{2}\, \delta^{2}\, \delta \,%
\acute{}%
\, \lambda=f^{2\,}\delta^{4\,}\lambda^{2}-\delta^{2}(f^{2}-\delta^{2})^{2}=0.
\]
We finally obtain $\delta \,%
\acute{}%
=\lambda=0$ and $f=\pm \, \delta$. Thus the surface $\Phi$ is a right
helicoid.\smallskip \newline The cases $n\geq2$ and $n\leq4$ lead to
contradictions, as one can easily confirm.\smallskip \newline The results are
formulated as follows:

\begin{proposition}
\textit{Let }$\Phi \subset%
\mathbb{R}
^{3}$\textit{ be a skew ruled }$C^{3}$\textit{-surface whose a principal
curvature has the form \emph{(\ref{6})}. Then one of the following
occurs:}\emph{\newline(a)}$\quad n=-1$\textit{, }$f(u)=-k(u)$\textit{ and
}$\Phi$\textit{ is an Edlinger-surface.\newline}\emph{(b)}$\quad
n=-2$\textit{, }$f(u)=\pm \, \delta(u)$\textit{ and }$\Phi$\textit{ is a right
helicoid.\newline}\emph{(c)}$\quad n=-3$\textit{, }$f(u)=\delta^{2}%
(u)\,k^{-1}(u)$\textit{ and }$\Phi$\textit{ is an Edlinger-surface.}
\end{proposition}

\noindent From this proposition follows the next

\begin{corollary}
\textit{Let }$\Phi \subset%
\mathbb{R}
^{3}$\textit{ be a skew ruled }$C^{3}$\textit{-surface, whose principal
curvatures satisfy the relation}\emph{ }%
\begin{equation}
\delta^{2\,}k%
\genfrac{}{}{0pt}{1}{3}{1}%
+k^{4}\,k_{2}=0.\label{9}%
\end{equation}
\textit{Then }$\Phi \ $\textit{is an Edlinger-surface}.
\end{corollary}

\begin{proof}
By using (\ref{4}) and (\ref{9}) we obtain $k%
\genfrac{}{}{0pt}{1}{4}{1}%
=k^{4}\,w^{-4}$, so that it is $k_{1}=\pm \,k\,w^{-1}$. From Proposition 1 it
follows $k_{1}=-k\,w^{-1}$. Thus the normal curvature $k_{1}$ has the required
form (\ref{6}), where $n=-1$. Hence $\Phi$ is an Edlinger-surface.
\end{proof}

\section{The case of the normal curvature}

Continuing this line of work we consider further geometrically distinguished
families of curves on the skew ruled surface $\Phi$ and suppose that the
normal curvature along the curves of these families are of the form
\begin{equation}
k_{N}=f(u)\,w^{n},\;n\in%
\mathbb{Z}
,\;f(u)\in C^{0}(I). \label{10}%
\end{equation}
Our aim is the specification of these ruled surfaces.\smallskip \medskip

\noindent \textbf{3.1. }Let $S_{1}$ \emph{be the family of curves of constant
striction distance}. From (\ref{7}) the normal curvature along a curve of
$S_{1}$ is obtained%
\[
k_{N}=\frac{1}{w}\cdot \frac{-k\,v^{2}-\delta \,%
\acute{}%
\,v+\delta^{2}\left(  k-\lambda \right)  }{v^{2}+\delta^{2}\left(  \lambda
^{2}+1\right)  }.
\]
Therefore $k_{N}$ has the form (\ref{10}) if and only if%
\begin{equation}
f\,w^{n+1\,}[v^{2}+\delta^{2}(\lambda^{2}+1)]+k\,v^{2}+\delta \,%
\acute{}%
\,v+\delta^{2}(k-\lambda)=0\text{.}\label{11}%
\end{equation}
It is observed that the function $f$ vanishes exactly when for all $v\in%
\mathbb{R}
$ holds%
\begin{equation}
k\,v^{2}+\delta \,%
\acute{}%
\,v+\delta^{2}(k-\lambda)=0,\label{12}%
\end{equation}
so that $k=\delta \,%
\acute{}%
\,=\lambda=0$, which means that $\Phi$ is a right helicoid.\smallskip \newline
Let now be $f\neq0$. For $n=-1$, using (\ref{11}) the following is derived
\[
Q(v):=f[v^{2}+\delta^{2}(\lambda^{2}+1)]+k\,v^{2}+\delta \,%
\acute{}%
\,v+\delta^{2}(k-\lambda)=0.
\]
The vanishing of the coefficients of the polynomial $Q(v)$ implies%
\[
f=-k,\quad \delta \,%
\acute{}%
\,=0,\quad \lambda \,(k\lambda+1)=0.
\]
Therefore the surface $\Phi$ is either an orthoid ruled surface\footnote{A
ruled surface is called \emph{orthoid} if its rulings are perpendicular to the
striction line.} of constant parameter of distribution ($\delta \,%
\acute{}%
=\lambda=0$) or an Edlinger-surface ($\delta \,%
\acute{}%
=k\lambda+1=0$).\smallskip \newline For $n>-1$ it follows from (\ref{11})%
\[
Q(v):=f^{2}(v^{2}+\delta^{2})^{n+1}[v^{2}+\delta^{2}(\lambda^{2}%
+1)]^{2}-[k\,v^{2}+\delta \,%
\acute{}%
\,v+\delta^{2}(k-\lambda)]^{2}=0.
\]
From the vanishing of the coefficient of $v^{2\left(  n+3\right)  }$\ it
follows that $f=0$ which it was excluded.\smallskip \newline For $n<-1$ it
follows from (\ref{11})%
\[
Q(v):=(v^{2}+\delta^{2})^{-n-1}[k\,v^{2}+\delta \,%
\acute{}%
\,v+\delta^{2}(k-\lambda)]^{2}-f^{2}[v^{2}+\delta^{2}(\lambda^{2}+1)]^{2}=0.
\]
The vanishing of the coefficient of $v^{2\left(  n-1\right)  }$ implies $k=0$.
Then the polynomial $Q(v)$ becomes%
\begin{equation}
Q(v)=(v^{2}+\delta^{2})^{-n-1}(\delta \,%
\acute{}%
\,v-\delta^{2}\, \lambda)^{2}-f^{2}[v^{2}+\delta^{2}(\lambda^{2}+1)]^{2}%
=0.\label{13}%
\end{equation}
For $n=-2$ the polynomial $Q(v)$ takes the form
\[
Q(v)=(v^{2}+\delta^{2})(\delta \,%
\acute{}%
\,v-\delta^{2}\lambda)^{2}-f^{2}[v^{2}+\delta^{2}(\lambda^{2}+1)]^{2}=0.
\]
From the vanishing of the coefficients of the polynomial $Q(v)$ the result is
$f=0$, which is a contradiction.\smallskip \newline For $n<-2$ the vanishing of
the coefficient of $v^{-2n}\ $in (\ref{13}) implies $\delta \,%
\acute{}%
=0$, therefore%
\begin{equation}
Q(v)=\delta^{4}\lambda^{2}(v^{2}+\delta^{2})^{-n-1}-f^{2}[v^{2}+\delta
^{2}(\lambda^{2}+1)]^{2}=0.\label{14}%
\end{equation}
In particular for $n=-3$ one has%
\[
Q(v)=\delta^{4}\lambda^{2}(v^{2}+\delta^{2})^{2}-f^{2}[v^{2}+\delta
^{2}(\lambda^{2}+1)]^{2}=0.
\]
From the vanishing of the coefficients of the polynomial $Q(v)$ it follows
again that $f=0$ which is a contradiction.\smallskip

\noindent For $n<-3$ (\ref{14}) results in $\lambda=f=0$ which is equally
impossible.\smallskip \newline Thus the following has been shown:

\begin{proposition}
\textit{Suppose that the normal curvature along the curves of constant
striction distance of a skew ruled }$C^{3}$\textit{-surface }$\Phi \subset%
\mathbb{R}
^{3}$\textit{ has the form \emph{(\ref{10})}. Then one of the following
occurs:}\emph{\newline(a)\quad}$f=0$\textit{ and }$\Phi$\textit{ is a right
helicoid.}\emph{\newline(b)\quad}$n=-1,f(u)=-k(u)$\textit{ and }$\Phi$\textit{
is either an orthoid ruled surface of constant parameter of distribution or an
Edlinger-surface.}\emph{\medskip}
\end{proposition}

\noindent \textbf{3.2.} Let $S_{2}$ be \emph{the family of the orthogonal
trajectories of the family }$S_{1}$. This family is determined by
\[
\lbrack v^{2}+\delta^{2}(\lambda^{2}+1)]\,du+\delta \, \lambda \,dv=0.
\]
From (\ref{7}) the corresponding normal curvature is obtained%
\[
k_{N}=\frac{1}{w^{3}}\cdot \frac{-\delta^{2}\lambda \, \left[  \left(  k\,
\lambda+2\right)  v^{2}+\delta \,%
\acute{}%
\, \lambda \,v+\delta^{2}\left(  \lambda^{2}+k\, \lambda+2\right)  \right]
}{v^{2}+\delta^{2}\left(  \lambda^{2}+1\right)  }.
\]
Consequently $k_{N}$ has the form (\ref{10}) if and only if%
\begin{equation}
f\,w^{n+3}[v^{2}+\delta^{2}(\lambda^{2}+1)]+\delta^{2}\lambda \lbrack(k\,
\lambda+2)v^{2}+\delta \,%
\acute{}%
\, \lambda \,v+\delta^{2}(\lambda^{2}+k\, \lambda+2)]=0. \label{15}%
\end{equation}
Obviously the function $f$ vanishes exactly when for all $v\in%
\mathbb{R}
$ holds
\[
\lambda \,[(k\, \lambda+2)v^{2}+\delta \,%
\acute{}%
\, \lambda \,v+\delta^{2}(\lambda^{2}+k\, \lambda+2)]=0.
\]
In this case the function $\lambda$ vanishes too, because otherwise it would
have been%
\[
k\, \lambda+2=\delta \,%
\acute{}%
\, \lambda=\delta^{2}(\lambda^{2}+k\, \lambda+2)=0,
\]
which are impossible. Therefore it is $f=0$ if and only if $\lambda=0$ and
this means that $\Phi$ is\ an orthoid ruled surface.\smallskip \newline Let now
be $f\, \lambda \neq0.$ For $n=-3$ it follows from (\ref{15})%
\[
Q(v):=f\,[v^{2}+\delta^{2}(\lambda^{2}+1)]+\delta^{2}\lambda \,[(k\,
\lambda+2)v^{2}+\delta \,%
\acute{}%
\, \lambda \,v+\delta^{2}(\lambda^{2}+k\, \lambda+2)]=0.
\]
From the vanishing of the coefficients of the polynomial $Q(v)$ it is obtained
that%
\[
f=-\delta^{2}\lambda \,(k\, \lambda+2),\quad \delta \,%
\acute{}%
=0,\quad k\, \lambda+1=0.
\]
This results in $f=-\delta^{2}\, \lambda$ and $\Phi$ is an
Edlinger-surface.\smallskip \newline One can easily confirm that the cases
$n>-3$ and $n<-3$ lead to contradictions. These results imply the following

\begin{proposition}
\textit{Suppose that the normal curvature along the orthogonal trajectories of
the curves of constant striction distance of a skew ruled }$C^{3}%
$\textit{-surface }$\Phi \subset%
\mathbb{R}
^{3}$\textit{ has the form \emph{(\ref{10})}. Then one of the following
occurs:}\emph{\newline(a)}$\quad f=0$\textit{ and }$\Phi$\textit{ is an
orthoid ruled surface.}\emph{\newline(b)}$\quad n=-3$\textit{, }%
$f(u)=\delta^{2}(u)\,k^{-1}(u)$\textit{ and }$\Phi$\textit{ is an
Edlinger-surface.}\emph{\medskip}
\end{proposition}

\noindent \textbf{3.3.} Let $S_{3}$ be the family \emph{of the orthogonal
trajectories of the rulings}, i.e. the family which is determined by%
\[
\delta \, \lambda \,du+dv=0.
\]
From (\ref{7}) the corresponding normal curvature is obtained%
\[
k_{N}=\frac{-k\,v^{2}-\delta \,%
\acute{}%
\,v-\delta^{2}\left(  k+\lambda \right)  }{w^{3}}.
\]
Therefore $k_{N}$ has the form (\ref{10}) if and only if
\begin{equation}
f\,w^{n+3}+k\,v^{2}+\delta \,%
\acute{}%
\,v+\delta^{2}(k+\lambda)=0. \label{16}%
\end{equation}
The function $f$ vanishes if and only if (\ref{12}) holds for all $v\in%
\mathbb{R}
$ or, equivalently, if $k=\lambda=\delta \,%
\acute{}%
=0$. Hence the surface $\Phi$ is a right helicoid.\smallskip \newline Let now
be $f\neq0$. For $n=-3$ it follows from (\ref{16})%
\[
Q(v):=k\,v^{2}+\delta \,%
\acute{}%
\,v+\delta^{2}(k+\lambda)+f=0.
\]
The vanishing of the coefficients of the polynomial $Q(v)$ implies%

\[
f=-\delta^{2}\, \lambda,\quad k=0,\quad \delta \,%
\acute{}%
=0.
\]
Therefore $\Phi$ is a conoidal ruled surface of constant parameter of
distribution.\smallskip \newline For $n=-2$ it follows from (\ref{16})%
\[
Q(v):=f^{2}(v^{2}+\delta^{2})-[k\,v^{2}+\delta \,%
\acute{}%
\,v+\delta^{2}(k+\lambda)]^{2}=0.
\]
From the vanishing of the coefficients of the polynomial $Q(v)$ it follows%
\[
k=0,\quad f^{2}=\delta \,%
\acute{}%
\,^{2},\quad \delta \,%
\acute{}%
\, \lambda=0,\quad f^{2}=\delta^{2}\lambda^{2}=0,
\]
therefore $f=0$ which it was excluded.\smallskip \newline From (\ref{16}) and
for $n=-1$ the outcome is
\[
Q(v):=f(v^{2}+\delta^{2})+k\,v^{2}+\delta \,%
\acute{}%
\,v+\delta^{2}(k+\lambda)=0.
\]
From the vanishing of the coefficients of the polynomial $Q(v)$ it follows
\[
f=-k,\quad \delta \,%
\acute{}%
=0,\quad \lambda=0.
\]
Therefore $\Phi$ is an orthoid ruled surface of constant parameter of
distribution.\smallskip \newline One can easily confirm that the cases $n\geq0$
and $n\leq-4$ lead to contradictions. So it can be stated that:

\begin{proposition}
\textit{Suppose that the normal curvature along the orthogonal trajectories of
the rulings of a skew ruled }$C^{3}$\textit{-surface }$\Phi \subset%
\mathbb{R}
^{3}$\textit{ has the form \emph{(\ref{10})}. Then one of the following
occurs:}\emph{\newline(a)}$\quad f=0$\textit{ and }$\Phi$\textit{ is a right
helicoid.}\emph{\newline(b)}$\quad n=-3$\textit{, }$f(u)=-\delta
^{2}(u)\, \lambda(u)$\textit{ and }$\Phi$\textit{ is a conoidal ruled surface
of constant parameter of distribution.}\emph{\newline(c)}$\quad n=-1$\textit{,
}$f(u)=-k(u)$\textit{ and }$\Phi$\textit{ is an orthoid ruled surface of
constant parameter of distribution.}\emph{\medskip}
\end{proposition}

\noindent \textbf{3.4.} Let $S_{4}$ be \emph{the family of curves of constant
Gaussian curvature}\textit{ }\cite{Sachs}. This family is determined by%
\[
\delta \,%
\acute{}%
\,(\delta^{2}-v^{2})\,du+2\delta \,v\,dv=0.
\]
Putting for abbreviation
\[
A=\left(  4\delta^{2}+\delta \,%
\acute{}%
\,^{2}\right)  v^{4}+4\delta^{2\,}\delta \,%
\acute{}%
\lambda \,v^{3}+2\delta^{2}\left[  2\delta^{2}\left(  \lambda^{2}+1\right)
-\delta \,%
\acute{}%
\,^{2}\right]  v^{2}-4\delta^{4\,}\delta \,%
\acute{}%
\, \lambda \,v+\delta^{4\,}\delta \,%
\acute{}%
\,^{2},
\]
the corresponding normal curvature can be computed from (\ref{7}) to be%
\[
k_{N}=\frac{-1}{w}\cdot \frac{4\delta^{2}\,v\left[  k\,v^{3}+\delta^{2}\left(
k-\lambda \right)  v+\delta^{2}\, \delta \,%
\acute{}%
\, \right]  }{A}.
\]
$k_{N}$ has the form (\ref{10}) if and only if%
\begin{align}
&  \text{ }f\,w^{n+1}[(4\delta^{2}+\delta \,%
\acute{}%
\,^{2})v^{4}+4\delta^{2}\, \delta \,%
\acute{}%
\, \lambda \,v^{3}+2\delta^{2}[2\delta^{2}(\lambda^{2}+1)-\delta \,%
\acute{}%
\,^{2}]v^{2}\nonumber \label{18}\\
&  -4\delta^{4\,}\delta \,%
\acute{}%
\, \lambda \,v+\delta^{4\,}\delta \,%
\acute{}%
\,^{2}]+4\delta^{2}v\,[k\,v^{3}+\delta^{2}(k-\lambda)v+\delta^{2}\delta \,%
\acute{}%
\,]=0.
\end{align}
The function $f$ vanishes exactly when%
\[
k\,v^{3}+\delta^{2}(k-\lambda)v+\delta^{2\,}\delta \,%
\acute{}%
\,=0
\]
for all $v\in%
\mathbb{R}
$ or, equivalently, if $k=\lambda=\delta \,%
\acute{}%
\,=0$. Consequently $\Phi$ is a right helicoid.\smallskip \newline Let now be
$f\neq0$. For $n=-1$ it follows from (\ref{18})%
\begin{align*}
Q(v)  &  :=f\,[(4\delta^{2}+\delta \,%
\acute{}%
\,^{2})v^{4}+4\delta^{2}\, \delta \,%
\acute{}%
\, \lambda \,v^{3}+2\delta^{2\,}[2\delta^{2}(\lambda^{2}+1)-\delta \,%
\acute{}%
\,^{2}]v^{2}\\
&  -4\delta^{4}\, \delta \,%
\acute{}%
\, \lambda \,v+\delta^{4}\delta \,%
\acute{}%
\,^{2}]+4\delta^{2}\,v[k\,v^{3}+\delta^{2}(k-\lambda)v+\delta^{2\,}\delta \,%
\acute{}%
\,]=0.
\end{align*}
The coefficients
\begin{align*}
a_{4}  &  :=f\,(4\delta^{2}+\delta \,%
\acute{}%
\,^{2})+4\delta^{2}k,\quad a_{3}:=4f\, \delta^{2}\, \delta \,%
\acute{}%
\, \lambda,\\
a_{2}  &  :=2f\, \delta^{2}[2\delta^{2}(\lambda^{2}+1)-\delta \,%
\acute{}%
\,^{2}]+4\delta^{4}(k-\lambda),\\
a_{1}  &  :=-4f\, \delta^{4}\delta \,%
\acute{}%
\, \lambda+4\delta^{4}\, \delta \,%
\acute{}%
\,,\quad a_{0}:=f\, \delta^{4\,}\delta \,%
\acute{}%
\,^{2}%
\end{align*}
of the polynomial $Q(v)$ vanish. From $a_{0}=0$ it follows $\delta \,%
\acute{}%
=0$. Then from the vanishing of the coefficients $a_{2}$ and $a_{4}$ we
obtain
\[
f=-k,\quad \lambda \,(k\, \lambda+1)=0.
\]
Consequently $\Phi$ is either an orthoid ruled surface of constant parameter
of distribution ($\delta \,%
\acute{}%
=\lambda=0$) or an Edlinger-surface ($\delta \,%
\acute{}%
=k\, \lambda+1=0$).\smallskip \newline The cases $n>-1$ and $n<-1$ lead to
contradictions. The following has been shown

\begin{proposition}
\textit{Suppose that the normal curvature along the curves of constant
Gaussian curvature of a skew ruled }$C^{3}$\textit{-surface }$\Phi \subset%
\mathbb{R}
^{3}$\textit{ has the form \emph{(\ref{10})}. Then one of the following
occurs:}\newline \emph{(a)}$\quad f=0$\textit{ and }$\Phi$\textit{ is a right
helicoid.}\emph{\newline(b)}$\quad n=-1$\textit{, }$f(u)=-k(u)$\textit{ and
}$\Phi$\textit{ is either an orthoid ruled surface of constant parameter of
distribution or an Edlinger-surface.}
\end{proposition}

The following table assemble the above results.\bigskip \newline%
\begin{tabular}
[c]{|c|c|c|c|}\hline
$%
\begin{tabular}
[c]{c}%
{\small Normal curvature of the}\\
{\small form }${\small k}_{{\small N}}$ ${\small =}$ ${\small f\,w}^{n}$
{\small along}%
\end{tabular}
$ & ${\small f}$ & $n$ & {\small Type of the ruled surface }${\small \Phi}%
$\\ \hline \hline%
\begin{tabular}
[c]{c}%
{\small one family of}\\
{\small the lines of curvature}%
\end{tabular}
&
\begin{tabular}
[c]{c}%
${\small -k}$\\
${\small \pm \delta}$\\
${\small \delta}^{2}{\small k}^{-1}$%
\end{tabular}
&
\begin{tabular}
[c]{c}%
${\small -1}$\\
${\small -2}$\\
${\small -3}$%
\end{tabular}
&
\begin{tabular}
[c]{c}%
$\cdot$\ {\small Edlinger-surface}\\
$\cdot$\ {\small right helicoid}\\
$\cdot$ {\small Edlinger-surface}%
\end{tabular}
\\ \hline%
\begin{tabular}
[c]{c}%
{\small the curves of const.}\\
{\small striction distance}%
\end{tabular}
&
\begin{tabular}
[c]{c}%
${\small 0}$\\
${\small -k}$%
\end{tabular}
&
\begin{tabular}
[c]{c}%
{\small -}\\
${\small -1}$%
\end{tabular}
& $%
\begin{array}
[c]{c}%
{\small \cdot}\text{ {\small right helicoid}}\\
{\small \cdot}\text{\ {\small either an orthoid surface of}}\\
\text{{\small const. parameter of distrib.}}\\
\text{{\small or an Edlinger-surface}}%
\end{array}
$\\ \hline%
\begin{tabular}
[c]{c}%
{\small the orthogonal trajectories}\\
{\small of the curves of const.}\\
{\small striction distance}%
\end{tabular}
&
\begin{tabular}
[c]{c}%
${\small 0}$\\
${\small \delta}^{2}{\small k}^{-1}$%
\end{tabular}
&
\begin{tabular}
[c]{c}%
{\small -}\\
${\small -3}$%
\end{tabular}
&
\begin{tabular}
[c]{c}%
$\cdot$\ {\small orthoid surface}\\
$\cdot$\ {\small Edlinger-surface}%
\end{tabular}
\\ \hline%
\begin{tabular}
[c]{c}%
{\small the orthogonal trajectories}\\
{\small of the rulings}%
\end{tabular}
&
\begin{tabular}
[c]{c}%
${\small 0}$\\
${\small -k}$\\
${\small -\delta}^{2}{\small \lambda}$%
\end{tabular}
&
\begin{tabular}
[c]{c}%
{\small -}\\
${\small -1}$\\
${\small -3}$%
\end{tabular}
&
\begin{tabular}
[c]{c}%
$\cdot$\ {\small right helicoid}\\
$\cdot$\ {\small orthoid surface of}\\
{\small const. parameter of distrib.}\\
$\cdot$\ {\small conoidal surface of}\\
{\small const. parameter of distrib.}%
\end{tabular}
\\ \hline%
\begin{tabular}
[c]{c}%
{\small the curves of const.}\\
{\small Gaussian curvature}%
\end{tabular}
&
\begin{tabular}
[c]{c}%
${\small 0}$\\
${\small -k}$%
\end{tabular}
&
\begin{tabular}
[c]{c}%
{\small -}\\
${\small -1}$%
\end{tabular}
& $%
\begin{array}
[c]{c}%
{\small \cdot}\text{ {\small right helicoid}}\\
{\small \cdot}\text{\ {\small either an orthoid surface of}}\\
\text{{\small const. parameter of distrib.}}\\
\text{{\small or an Edlinger-surface}}%
\end{array}
$\\ \hline
\end{tabular}
\linebreak \  \newline

\end{document}